%% file: topological_3D.tex
\documentclass[12pt, twoside]{article}

\usepackage{tikz}
\usetikzlibrary{calc}

\usepackage{amsmath,amsthm}

\usepackage[utf8]{inputenc}

\usepackage{enumerate}
\usepackage{esint}

\usepackage{fancyhdr}
\usepackage{cite}
\usepackage{enumitem}

\usepackage{xargs}                      
\usepackage[colorinlistoftodos,prependcaption,textsize=tiny]{todonotes}
\newcommandx{\pcomment}[2][1=]{\todo[linecolor=red,backgroundcolor=red!25,bordercolor=red,#1]{#2}}
\newcommandx{\kcomment}[2][1=]{\todo[linecolor=blue,backgroundcolor=blue!25,bordercolor=blue,#1]{#2}}

\newif\ifpreprint

\usepackage{hyperref}
\hypersetup{%
    colorlinks=True, 
    citecolor=blue,
    linkcolor=black}

\usepackage[charter]{mathdesign}

\usepackage{accents}

\newcommand{\rom}[1]{\uppercase\expandafter{\romannumeral #1\relax}}

  \theoremstyle{definition}
  \newtheorem{theorem}{Theorem}[section]
  \newtheorem{corollary}[theorem]{Corollary}
  
  \newtheorem{lemma}[theorem]{Lemma}
  \newtheorem{definition}[theorem]{Definition}
  \newtheorem{remark}[theorem]{Remark}

  \newtheorem*{maintheorem}{Main Theorem}

\newtheorem{thmx}{Assumption}
\newtheorem{assumption}[thmx]{Assumption}

  \newtheorem*{assumption*}{Assumption}

  \newcommand{\fu}{u}
  \newcommand{\fp}{p}
  
  \numberwithin{equation}{section}
    
  \newcommand{\subjclass}[1]{\bigskip\noindent\emph{2010 Mathematics Subject Classification:}\enspace#1}
  \newcommand{\keywords}[1]{\noindent\emph{Keywords:}\enspace#1}

\input{slidesymbols}

\setenumerate[0]{label=(\alph*)}

\newcommand{\eps}{{\varepsilon}}

\newcommand{\overbar}[1]{\mkern 1.5mu\overline{\mkern-1.5mu#1\mkern-1.5mu}\mkern 1.5mu}

\newcommand{\ben}{\begin{equation}}
\newcommand{\een}{\end{equation}}

\newcommand{\benn}{\begin{equation*}}
\newcommand{\eenn}{\end{equation*}}

\usepackage[colorinlistoftodos]{todonotes}

\newcommand{\tow}{\rightharpoonup}

\usepackage{authblk}

\begin{document}

\title{A simplified derivation technique of topological derivatives for quasi-linear transmission problems}

\author{Peter Gangl\footnote{TU Graz, Steyrergasse 30/III, 8010 Graz, Austria, gangl(at)math.tugraz.at } $\;$ and 
        Kevin Sturm\footnote{TU Wien, Wiedner Hauptstr. 8-10,
      1040 Vienna, Austria, E-Mail: kevin.sturm(at)tuwien.ac.at}}

\date{\today}

\maketitle

\begin{abstract}
 
In this paper we perform the rigorous derivation of the topological derivative for optimization problems constrained by a class of quasi-linear elliptic transmission problems. 
In the case of quasi-linear constraints, techniques using fundamental solutions of the differential operators cannot be applied to show convergence of the variation of the states. Some authors succeeded showing this convergence with the help of technical computations under additional requirements on the problem.
Our main objective is to simplify and extend these previous results by using a Lagrangian framework and a projection trick. Besides these generalisations the purpose of this manuscript is to present a systematic derivation approach for topological derivatives.


\end{abstract}

  \maketitle

\section{Introduction}

The topological derivative of a shape functional $J = J(\Omega)$, where $\Omega \subset \VR^d$, measures the sensitivity of the functional with respect to a topological perturbation of the shape $\Omega$. The concept was first used in \cite{EschenauerKobelevSchumacher1994} in the context of linearized elasticity as a means to find optimal locations for introducing holes into an elastic structure. Later, the concept was introduced in a mathematically rigorous way in \cite{SokolowskiZochowski1999}. 
In the literature many research articles deal with the derivation of topological sensitivities of optimization problems which are constrained by linear partial differential equations (PDEs). We refer the reader to \cite{Amstutz2006} as well as the monograph \cite[pp. 3]{b_NOSO_2013a} and references therein. The topological derivative for a class of semilinear PDEs with the Laplace operator as the principal part was studied in \cite{Amstutz2006b, MR2541192}, and more recently in \cite{Sturm2019} using an averaged adjoint framework.

As it is mentioned in the recent book \cite[Sec. 6.4, p.107]{NovotnySokolowskiZochowski2019}, 
\begin{quote}
    ``Extension to nonlinear problems in general can be considered the main challenge
in the theoretical development of the topological derivative method. The difficulty
arises when the nonlinearity comes from the main part of the operator, which at
the same time suffers a topological perturbation.''
\end{quote} 
This statement applies in particular to quasi-linear PDEs when the main part of the differential operator gets topologically perturbed. In this case, techniques based on fundamental solutions, as they are heavily used in the linear and semi-linear case, cannot be applied any more and other strategies have to be followed.

The first rigorous results of topological sensitivity analysis for shape functions constrained by quasi-linear PDEs were obtained in \cite{a_AMBO_2017a} where the authors consider a regularized version of the $p$-Poisson equation. Based on these results, the topological derivative for the quasi-linear equation of 2D magnetostatics was derived in \cite{AmstutzGangl2019} where also the numerical treatment of the obtained formula was addressed.

In this paper, we establish the topological derivative for a  larger  class of quasi-linear problems under  more 
general assumptions. More precisely, given a fixed, open and bounded hold-all domain $\Dsf$ and an open and measurable subset $\Omega \subset \Dsf$, we study the topological sensitivity analysis of the tracking-type cost function
\begin{align} \label{defJ}
    J(\Omega) &=  \int_{\Dsf} | \nabla(u - u_d)|^2 \;  dx
\end{align}
subject to the constraint that $ u\in H^1_0(\Dsf)$ solves
\begin{align}\label{E:weakformulation}
    \int_{\Dsf} \mathcal A_\Omega(x, \nabla u ) \cdot \nabla \varphi \; dx = \int_{\Dsf} f\varphi \;dx \quad \text{ for all } \varphi\in H^1_0(\Dsf).
\end{align}
Here, $f\in L_2(\Dsf)$, $u_d\in H^1_0(\Dsf)$  and $\Ca_\Omega :\Dsf \times \VR^d \rightarrow \VR^d$ is a piecewise nonlinear function defined by
\ben \label{eq_AOmega}
\Ca_\Omega(x,y) := \left\{
  \begin{array}{cl}
    a_1(y) & \text{ for } x\in \Omega\\
a_2(y) & \text{ for } x\in \Dsf\setminus\Omega,
\end{array}\right.
\een
with $a_1,a_2:\VR^d\to \VR^d$ being functions satisfying monotonicity and continuity assumptions. 

The crucial ingredient for our result is the strong convergence (Theorem~\ref{T:Keps_strong_K})
of the variation of the direct states,
\ben\label{E:state_strong}
\nabla \left(\frac{(u_\eps - u_0)\circ T_\eps}{\eps}\right) \to \nabla K \quad \text{ strongly in } L_2(\VR^d)^d,
\een
where $u_\eps$ and $u_0$ correspond to the solutions to the perturbed and unperturbed state equation, respectively.
As shown in \cite{Sturm2019}, for semilinear problems only weak convergence in \eqref{E:state_strong} is necessary to establish the topological derivative. For quasi-linear problems we need the strong convergence \eqref{E:state_strong}. 
 The main contributions of this work are as follows:
\begin{itemize}
    \item simplified analysis for derivation of topological derivative for quasi-linear equations
    \item generalisation of previous results
    \item relaxation of smoothness assumption on inclusion $\omega$
\end{itemize}

%

The presented approach for deriving the topological derivative under a quasi-linear PDE constraint simplifies and also generalizes the approaches presented in \cite{a_AMBO_2017a} and \cite{AmstutzGangl2019} which is the subject of the following discussion.

The main difference between the presented approach and the results obtained in \cite{a_AMBO_2017a, AmstutzGangl2019} lies in the technique used to show \eqref{E:state_strong}, i.e. the strong convergence of the variation of the state on the rescaled bounded domain to the solution $K$ of a transmission problem on the unbounded domain as $\eps \to 0$. While, in our approach, this convergence is accomplished by the introduction of a projection $\hat K_\eps$ of $K$ into the space $H_0^1(\eps^{-1}\Dsf)$, the main ingredient used in \cite{a_AMBO_2017a, AmstutzGangl2019} is a cut-off argument relying on explicit knowledge of the asymptotic behavior of $K$ as $|x| \to \infty$. The authors successfully showed the necessary decay of $K$ by a comparison principle. However, this was achieved using long, technical calculations which additionally required stronger assumptions on the data compared to what is presented here.

In particular, the authors of \cite{a_AMBO_2017a, AmstutzGangl2019} have to assume that $\omega = B_1(0)$ is the unit ball, whereas our approach remains valid for any open and bounded set $\omega \subset \VR^d$ with $0 \in \omega$. Furthermore, the assumptions on the class of quasi-linear PDEs used here, i.e. Assumption \ref{A:nonlinearity}, are less restrictive than those used in \cite{a_AMBO_2017a, AmstutzGangl2019}. In particular, the proof technique used there requires the third derivative of the operator $a_i$ to be bounded (see p.75 in \cite{Bonnafe2013} or Assumption 3.3. in \cite{AmstutzGangl2019}) and, in the case of \cite{AmstutzGangl2019}, an additional nonphysical assumption on the materials which is not necessarily satisfied in practice (see Assumption 3.4 in \cite{AmstutzGangl2019}). We remark that the setting of \cite{AmstutzGangl2019} is fully covered in our analysis without this nonphysical assumption. 

Moreover, in both \cite{a_AMBO_2017a} and \cite{AmstutzGangl2019}, the case $z \in \Omega$ and $z \in \Dsf \setminus \overbar{\Omega}$ have to be treated separately. This is not addressed in \cite{a_AMBO_2017a} where the proof relies on $\gamma_1 < \gamma_0$, and is carried out by repeating adaptations of the technical proofs in the setting of \cite{AmstutzGangl2019}, see also \cite[Sec. 4.5]{GanglDiss2017}. In the approach presented here, it is enough to interchange the roles of $a_1$ and $a_2$ to get to the topological derivative for the other scenario.

Finally, we will also show how to treat objective functionals of the form 
\begin{equation*}
    \int_D |u-u_d|^2 \, \mbox dx,
\end{equation*}
in Section~\ref{sec_averaged_adjoint}, which is not covered by the analysis shown in \cite{a_AMBO_2017a} and \cite{AmstutzGangl2019}.

\text{}\newline
The rest of this paper is organized as follows: In Section~\ref{sec_assumpMainRes} we state the main assumptions and the main result. The remaining sections are devoted to the proof of this result. In Section~\ref{sec_adjFramework}, we recall and extend results from an abstract Lagrangian framework that will be used to derive the topological derivative. In Section~\ref{sec_topDer} we show that the hypotheses of the abstract theorem are satisfied and obtain the final formula.  In Section~\ref{sec_averaged_adjoint} we compare the Lagrangian framework of Section~\ref{sec_adjFramework} with the averaged framework. 


\section{Assumptions and main results} \label{sec_assumpMainRes}

\subsection{Preliminaries: notation and definitions}

\paragraph{Function spaces}
Standard $L^p$ spaces and Sobolev spaces on an open set $\Dsf\subset \VR^d$ are denoted $L_p(\Dsf)$ and $W^k_p(\Dsf)$, respectively, where $p\ge 1$ and $k\ge 1$. In case $p=2$ and $k\ge 1$  we set as usual $H^k(\Dsf):= W^k_2(\Dsf)$.   Vector valued spaces are denoted $L_p(\Dsf)^d:=L_p(\Dsf,\VR^d)$ and $W^k_p(\Dsf)^d:=W^k_p(\Dsf,\VR^d)$.  We denote by $H^1_0(\Dsf)$ the subspace of functions in $H^1(\Dsf)$ with vanishing trace on $\partial\Dsf$. Given a normed vector space $V$ we denote by $\mathcal L(V,\VR)$ the space of linear and continuous functions on $V$.  We denote by $B_\delta(x)$ the ball centred at $x$ with radius $\delta >0$ and  set $\bar B_\delta(x) :=\overbar{B_\delta(x)}$. For the ball centered at $x=0$ we write $B_\delta:= B_\delta(0)$. 

For $d\ge 1$  and $1\le p <\infty$,  we set  $BL_p(\VR^d) :=  \{u\in W^1_{p,\text{loc}}(\VR^d):\; \nabla u \in L_p(\VR^d)^d\}$ 
and define the \emph{Beppo-Levi space} as the quotient space  $\dot{BL}_p(\VR^d) := BL_p(\VR^d)/\VR$, where $/\VR$ means that we quotient out the constant functions. We denote by $[u]$ the equivalence classes of $\dot{BL}(\VR^d)$. Equipped with the norm
      \ben
        \|[u]\|_{ \dot{BL}_p(\VR^d) } := \|\nabla u\|_{ L_p(\VR^d)^d}, \quad u\in [u],
      \een
the Beppo-Levi space is a  Banach space  (see \cite{a_DELI_1955a,a_ORSU_2012a}) and $C^\infty_c(\VR^d)/\VR$ is dense in $\dot{BL}_p(\VR^d)$.  In case $p=2$ the space $\dot{BL}(\VR^d)$ becomes a Hilbert space and we abbreviate this space simply with $\dot{BL}(\VR^d)$.

  Moreover, we write $\fint_A f\; dx := \frac{1}{|A|} \int_{A} f\; dx$ to indicate the average of $f$ over a measurable set $A$ with measure $|A|<\infty$. 
  We equip $\VR^d$ with the Euclidean norm $| \cdot |$ and use the same notation for the corresponding matrix (operator) norm.

 \paragraph{Definition of topological derivative}
 Before we state our main result we recall the definition of the topological derivative. 
 We restrict ourselves to the special case as it was introduced in \cite{SokolowskiZochowski1999} and refer the reader to \cite[pp. 4]{b_NOSO_2013a} for the more general definition.
\begin{definition}[Topological derivative]
    Let $\Dsf\subset \VR^3$ be an open set and $\Omega\subset \Dsf$ an open subset.  Let $\omega\subset \VR^3$ be open with $0\in \omega$. Define for $z\in \VR^3$,  $\omega_\eps(z) := z+\eps \omega$. 
    Then the topological derivative of $J$ at $\Omega$ at the point $z\in \Dsf \setminus \partial \Omega$ is defined by 
		\ben \label{def_TD}
		dJ(\Omega)(z) =  \left\{\begin{array}{ll}
                \lim_{\eps\searrow 0}\frac{J(\Omega\setminus \omega_\eps(z)) - J(\Omega)}{|\omega_\eps(z)|} & \text{ if } z \in \Omega, \\
                \lim_{\eps\searrow 0}\frac{J(\Omega\cup  \omega_\eps(z)) - J(\Omega)}{|\omega_\eps(z)|} & \text{ if } z\in  \Dsf \setminus \overbar \Omega.
		\end{array}\right.
		\een
\end{definition}
Without loss of generality, we will restrict ourselves to the second case and will always assume $ z \in D\setminus \overbar \Omega$. The derivation for the case $z \in  \Omega$ is analogous, cf. Remark \ref{rem_z}.

\subsection{Main results}\label{subsec_main_results}

We need the following assumptions:
\begin{assumption}\label{A:nonlinearity}
    There are constants $c_1, c_2, c_3$ such that the functions $a_i:\VR^d\to \VR^d$, $i=1,2$ are differentiable and satisfy:
    \begin{itemize}
        \item[(i)] 
            $(a_i(x)-a_i(y))\cdot (x-y) \ge  c_1 |x-y|^2   \quad \text{ for all } x,y\in \VR^d.$
        \item[(ii)] 
           $ |a_i(x)-a_i(y)|\le  c_2|x-y|  \quad \text{ for all } x,y\in \VR^d.$
        \item[(iii)] 
            $  |\partial a_i(x)-\partial a _i(y)|\le c_3 |x-y|  \quad \text{ for all } x,y\in \VR^d.$
    \end{itemize}
\end{assumption}

\begin{remark}
    By using the inverse triangle inequality and choosing $y=0$, we get from Assumption \ref{A:nonlinearity}(ii) and (iii) that
    \begin{align} 
        | a_i(x) | &\leq |a_i(0)| + c_2 |x|, \label{rem_ai} \\
        | \partial a_i(x) | & \leq |\partial a_i(0)| + c_3 |x|, \label{rem_aiii} 
    \end{align}
    for $i=1,2$ and for all $x \in \VR^d$. Notice also that using (ii), we get
    \ben
    |\partial a_i(x)v|= \lim_{t\searrow 0} |a_i(x+tv) - a_i(x)|/t\le c_2|v|,
    \een
for $i=1,2$ and  all $x,v\in \VR^d$.
\end{remark}

Properties (i) and (ii) of Assumption \ref{A:nonlinearity} imply that the operator $A_\Omega:H^1_0(\Dsf)\to (H^1_0(\Dsf))^*$ defined by $\langle A_\Omega \varphi,\psi\rangle := \int_\Dsf\Ca_\Omega(x,\nabla \varphi)\cdot \nabla \psi\;dx$ is Lipschitz continuous and strongly monotone for all measurable $\Omega  \subset \Dsf$. Hence the state equation \eqref{E:weakformulation} admits a unique solution by the theorem of Zarantonello; see \cite[p.504, Thm. 25.B]{b_ZE_1990a}.

    We restrict ourselves to Dirichlet boundary conditions but also other boundary conditions, e.g., Neumann boundary conditions, could be considered. In what follows at many places we extend functions $u\in H^1_0(\Omega)$  to function $\tilde u \in H^1(\VR^d)$ by setting $u$ to zero outside of $\Dsf$. When dealing with other boundary conditions, we can replace 
    this extension by the standard Sobolev extension operator $E:H^1(\Dsf) \to H^1(\VR^d)$.

Before we state our main result we introduce the adjoint $p\in H^1_0(\Dsf)$ as the solution to
\ben \label{eq_adjoint}
\int_\Dsf \partial_{u}\Ca_\Omega(x,\nabla u)(\nabla \varphi)\cdot \nabla p\;dx = -  \int_\Dsf 2\nabla (u-u_d)\cdot \nabla \varphi\;dx  \quad \text{ for all } \varphi\in H^1_0(\Dsf).
\een
In view of the monotonicity of $\Ca_\Omega$ the previous equation has according to Lax-Milgram a unique solution in $H^1_0(\Dsf)$.

\begin{figure}
    \centering \includegraphics[width=0.5\textwidth]{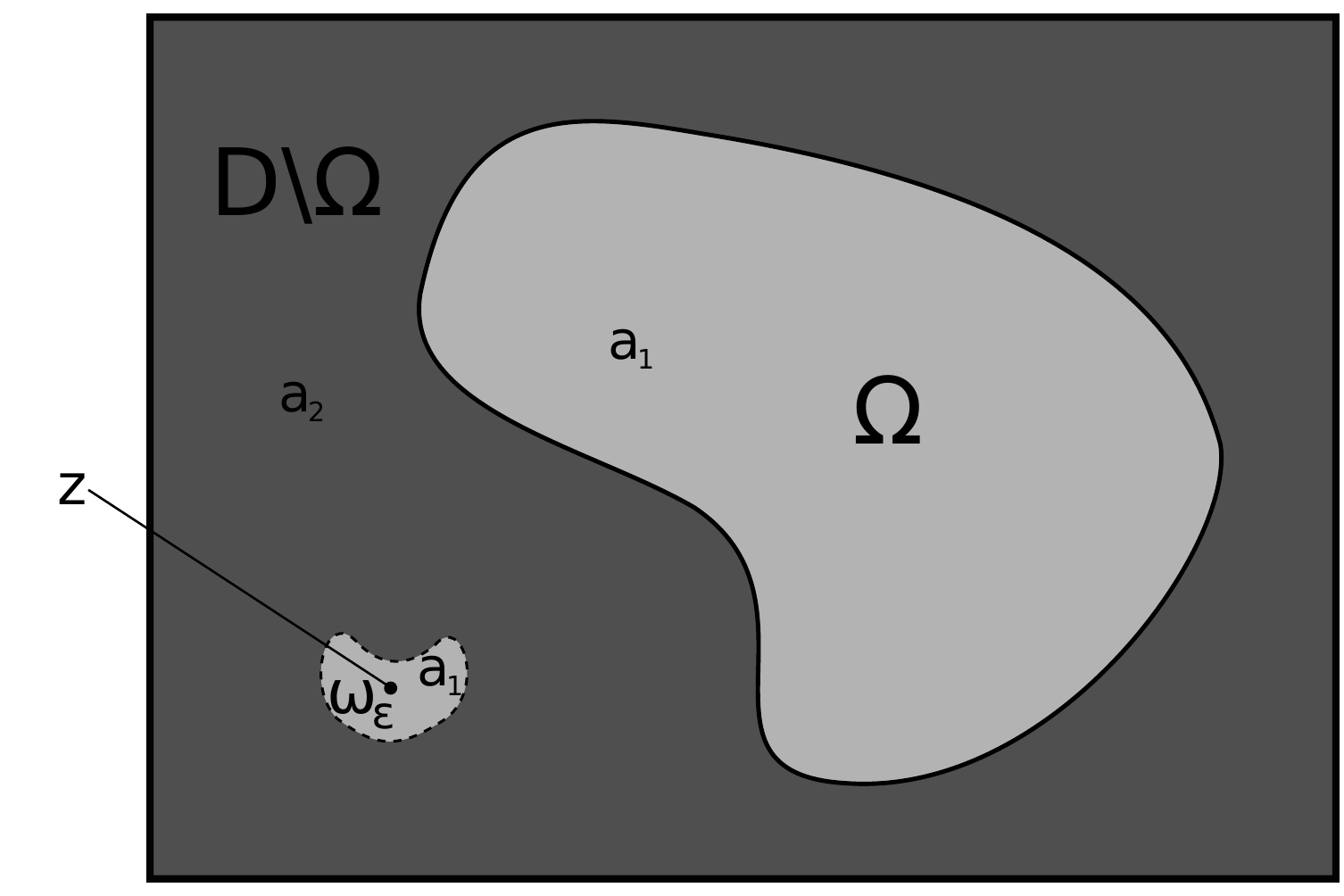}
    \caption{Setting for topological derivative: Inclusion $\omega_\eps$ of radius $\eps>0$ containing material $a_1$ around point $z \in \Dsf \setminus \overbar \Omega$ (where material $a_2$ is present).}

    \label{fig_setting}
\end{figure}

We fix the following setting for the topological perturbation (cf. Figure~\ref{fig_setting}):
\begin{itemize}\setlength\itemsep{0.3em}
  \item an open and bounded set $\omega\subset \VR^d$ with $0\in \omega$,
  \item an open set $\Omega \Subset \Dsf$ and the inclusion point $z:=0 \in \Dsf\setminus \overbar \Omega$,
  \item the perturbation $\omega_\eps(z) := \eps \omega$ and $\eps\in[0,\tau]$, where $\tau >0$ is such that $\omega_\eps(z) \Subset \Dsf\setminus \overbar \Omega$ for all $\eps \in [0,\tau]$.
  \item the perturbed shape $\Omega_\eps(z) := \Omega \cup \omega_\eps(z)$
  \item $T_\eps (x) := \eps x$, $x\in  \VR^d $, $\eps \ge 0$
\end{itemize}
To simplify notation we will often write $\omega_\eps$ instead of $\omega_\eps(z)$, $\Omega_\eps$ instead of $\Omega_\eps(z)$ and $x_\eps$ instead of $T_\eps(x)$. For $\eps>0$ we introduce the notation 
$\eps^{-1}\Dsf:= T_\eps^{-1}(\Dsf)$.

Let $\ell(\eps):= |\omega_\eps|$, and introduce the Lagrangian $ G:[0,\tau]\times H^1_0(\Dsf)\times H^1_0(\Dsf) \to \VR$ associated with the perturbation $\omega_\eps$ by 
\ben\label{eq:lagrange_scale_eps}
G(\eps, u, p) :=  \int_{\Dsf} | \nabla(u - u_d)|^2 \; dx +  \int_{\Dsf} \Ca_{\Omega_\eps}(x, \nabla u) \cdot \nabla p \; dx - \int_{\Dsf} f p \;dx.
\een
Here, the operator $\Ca_{\Omega_\eps}$ is defined according to \eqref{eq_AOmega} with $\Omega_\eps = \Omega \cup \omega_\eps$.

Now we can state our main result of this paper:
\begin{maintheorem}
Let Assumption~\ref{A:nonlinearity} be satisfied. Let $\Omega \subset D$ open and $u_0$ the solution to \eqref{E:weakformulation} and $p_0$ the solution to \eqref{eq_adjoint}. Let $z \in  \Dsf \setminus \overbar \Omega$ and assume that $u_0\in C^{1,\alpha}(\overline{B_\delta(z)})$ and $p_0\in C^1(\overline{B_\delta(z)})$ for some $\delta >0$ and $0 < \alpha<1$. Assume further that $\nabla p_0\in L_\infty(\Dsf)^d$. 
    \begin{itemize}
    \item[(a)] Then the assumptions of Theorem~\ref{thm:diff_lagrange} are satisfied for the Lagrangian $G$ given by \eqref{eq:lagrange_scale_eps} and hence the topological derivative at $z\in \Dsf \setminus \overbar \Omega$ is given by 
        \ben\label{eq:top_formula}
dJ(\Omega)(z) = \partial_{\ell} G(0, u_0, p_0) + R_1 ( u_0, p_0) + R_2(u_0,p_0) 
  \een
\item[(b)] 
We have 
\ben
\partial_{\ell} G(0, u_0, p_0) = ((a_1 (U_0) - a_2 (U_0)) \cdot P_0
\een
and 
\ben\label{E:R_term}
\begin{split}
    R_1( u_0, p_0) 
    = \frac{1}{|\omega|} \bigg(\int_{\VR^d} 
    \big[\Ca_\omega(x, \nabla K+U_0) - \Ca_\omega(x,U_0) - \partial_{u} \Ca_\omega(x,U_0 )& (\nabla K) 
               \big] 
           \cdot P_0\;dx\\ +\int_{\VR^d}|\nabla K|^2\;dx\bigg)
\end{split}
\een
and 
\ben
R_2(u_0,p_0)=          \frac{1}{|\omega|}\int_{\omega} \left[\partial_{u} a_1 (U_0) - \partial_{u} a_2 (U_0) \right] (\nabla K) \cdot P_0   \;dx 
\een
where $U_0 := \nabla u_0(z)$, $P_0:= \nabla p_0(z)$ and $\Ca_\omega(x,y):= a_1(y)\chi_\omega(x) + a_2(y)\chi_{\VR^d\setminus \omega}(x)$.
Here  $K \in \dot{BL}(\VR^d)$ is the unique solution to
    \ben
\begin{split}
    \int_{\VR^d} &  (\Ca_\omega(x, \nabla K+U_0  ) - \Ca_\omega(x, U_0 ))\cdot \nabla \varphi \; dx \\
                 &  = - \int_{\omega}(a_1 (U_0 ) - a_2 (U_0 )) \cdot \nabla \varphi \;dx \quad \text{ for all } \varphi\in BL(\VR^d).
\end{split}
\een 
\end{itemize}
\end{maintheorem}

\begin{remark} \label{rem_z}
    We restrict ourselves to the case where $z \in D \setminus \overbar \Omega$ without loss of generality. However, the exact same proof can be conducted in the case where $z \in  \Omega$. In that case, the formula for the topological derivative is obtained by just switching the roles of $a_1$ and $a_2$ in the theorem above (in particular also in the definition of $\Ca_\omega$).
    
    The assumption $z=0$ is without loss of generality, too. In the general case, this situation can be obtained by a simple change of the coordinate system.
\end{remark}

\begin{remark}
    Although we assume $f\in L_2(\Dsf)$, also more general right hand sides, such as
    $f_\Omega := \chi_\Omega f_1 + \chi_{\Dsf\setminus \Omega} f_2$ with $f_1,f_2\in L_2(\Dsf)$ could be considered with minor changes. 
\end{remark}

\begin{remark}
    We note that in \cite{a_AMBO_2017a} the topological derivative for a quasi-linear problem in $L_p$ spaces is considered. We believe that our analysis can also be transferred to this setting. 
\end{remark}

\begin{remark}
Although we did not treat the limiting case where $a_1$ or $a_2$ is zero, this can be done 
in a similar fashion. We refer to Section~5 in \cite{Sturm2019}. Dirichlet conditions on the inclusion using our approach have to be studied in different manner and deserve further research. 
\end{remark}

\section{Lagrangian framework}\label{sec_adjFramework}
In this section we recall results on a Lagrangian framework, which is a suitable 
refinement of \cite{Delfour2018}. These abstract results
will be used to derive the topological derivative for our quasi-linear model problem. We begin with the definition of a Lagrangian function; see also \cite{a_DEST_2017a}.

\begin{definition}[parametrised Lagrangian]
 Let $X$ and $Y$ be vector spaces and $\tau >0$. A parametrised Lagrangian (or short Lagrangian) is a function
\begin{gather*}
(\eps,\fu,\fp) \mapsto  G(\eps,\fu,\fp): [0, \tau ] \times X \times Y \to \VR,
\end{gather*}
satisfying, 
        \ben
\fp\mapsto G(\eps,\fu,\fp) \quad \text{ is }  \text{affine} \text{ on } Y.
\een
\end{definition}

\begin{definition}[state and adjoint state]
Let $\eps \in [0,\tau]$.  We define the state equation by: find $u_\eps \in X$, such that 
  \ben\label{eq:state}
   \partial_\fp G(\eps,\fu_\eps,0)(\varphi)=0 \quad \text{ for all } \varphi\in Y. 
\een
The set of states is denoted $E(\eps)$. We define the adjoint state by: find $p_\eps\in Y$, such that 
\ben\label{eq:adjoint_state}
   \partial_u G(\eps,\fu_\eps,\fp_\eps)(\varphi)=0 \quad \text{ for all } \varphi\in  X. 
\een
The set of adjoint states associated with $(\eps,u_\eps)$ is denoted $Y(\eps,u_\eps)$.
\end{definition}

\begin{definition}[$\ell$-differentiable Lagrangian]\label{def:lagrangian}
    Let $X$ and $Y$ be vector spaces and $\tau >0$. Let $\ell: [0,\tau] \to \VR$ be a given function satisfying $\ell(0)=0$ and $\ell(\eps)>0$ for 
$\eps \in (0,\tau]$. An $\ell$-differentiable parametrised Lagrangian is a parametrised Lagrangian $G:[0,\tau]\times X\times Y\to \VR$, satisfying,
\begin{itemize}
\item[(a)] for all $v,w\in X$ and $p\in Y$, 
    \ben\label{E:c1_lagrangian}
    s\mapsto G(\eps, v+s w, \fp) \text{ is continuously differentiable on } [0,1].
    \een
\item[(b)] for all $\fu_0\in E(0)$ and $\fp_0\in Y(0,\fu_0)$ the limit 
\ben
\partial_\ell G(0,\fu_0,\fp_0) := \lim_{\eps\searrow 0}\frac{G(\eps,\fu_0, \fp_0) - G(0,\fu_0, \fp_0)}{\ell(\eps)} \quad \text{ exists}.
\een
\end{itemize}
\end{definition}

\begin{assumption*}[H0]
    \begin{itemize}
     \item[(i)] We assume that for all $\eps \in [0,\tau]$, the set $E(\eps)=\{u_\eps\}$ is a singleton. 
     \item[(ii)] We assume that the adjoint equation for $\eps=0$, $\partial_u G(0,u_0,p_0)(\varphi)=0$ for all $\varphi \in E$, admits a unique solution.
    \end{itemize}
\end{assumption*}

We now give sufficient conditions when the function 
\ben\label{def_g}
\begin{split}
[0,\tau] &\to \VR \\
  \eps &\mapsto g(\eps) := G(\eps ,\fu_\eps,0),
  \end{split}
\een 
is one sided $\ell$-differentiable, that means, when the limit
\ben\label{E:dEll}
    d_{\ell}g(0):= \lim_{\eps\searrow 0}\frac{g(\eps)-g(0)}{\ell(\eps)}
    \een
    exists, where $\ell: [0,\tau] \to \VR$ is a given function satisfying $\ell(0)=0$ and $\ell(\eps)>0$ for $\eps \in (0,\tau]$. 
    
    The following theorem is a refinement of \cite[Thm. 3.3]{Delfour2018}. Instead of having one $R$-term we obtain two terms, which simplifies the later analysis. 
\begin{theorem} \label{thm:diff_lagrange}
    Let $G:[0,\tau]\times X\times Y \to \VR$ be an $\ell$-differentiable parametrised Lagrangian 
    satisfying Hypothesis~(H0). Define for $\eps > 0$, 
    \ben
    R_1^\eps (u_0,p_0):=   \frac{1}{\ell(\eps)}  \int_0^1 \left(\partial_u G(\eps,s u_\eps + (1-s)u_0, p_0) -   \partial_u G(\eps, u_0, p_0)\right)(u_\eps - u_0) \; ds
\een
and
\ben
\begin{split}
    R_2^\eps ( u_0 ,p_0 ) :=\frac{1}{\ell(\eps)}(\partial_u G(\eps, u_0,p_0) - \partial_u G(0,u_0,p_0))(u_\eps - u_0).
\end{split}
\een
If $R_1(u_0,p_0) := \lim_{\eps\searrow 0} R^\eps_1(u_0,p_0)$ and $R_2(u_0,p_0) := \lim_{\eps\searrow 0} R^\eps_2(u_0,p_0)$ exist, then  
\benn
d_\ell g(0) = \partial_\ell G(0,\fu_0,\fp_0) + R_1(u_0,p_0) + R_2(u_0,p_0).
\eenn
\end{theorem}
\begin{proof}
    Using $\partial_uG(0,u_0,p_0)(\varphi)=0$ for all $\varphi \in E$ and the fundamental theorem of calculus, we obtain
    \benn
    \begin{split}
        g(\eps) - g(0) & = G(\eps,u_\eps,p_0) - G(0,u_0,p_0) = G(\eps,u_\eps,p_0) - G(\eps,u_0,p_0)  +  G(\eps,u_0,p_0) - G(0,u_0,p_0) \\ 
        =& \int_0^1 \partial_u G(\eps,su_\eps + (1-s)u_0,p_0)(u_\eps-u_0)\; ds  +  G(\eps,u_0,p_0) - G(0,u_0,p_0) \\
        = & \int_0^1 ( \partial_u G(\eps,su_\eps + (1-s)u_0,p_0) - \partial_uG(\eps, u_0,p_0)) (u_\eps - u_0)\; ds \\
          & +  (\partial_u G(\eps,u_0,p_0) - \partial_u G(0,u_0,p_0))(u_\eps-u_0) \\
                       & +  G(\eps,u_0,p_0) - G(0,u_0,p_0).
\end{split}
    \eenn
    Notice that the fundamental theorem of calculus is applicable in view of assumption \eqref{E:c1_lagrangian}. Now dividing by $\ell(\eps)$, using Hypothesis~(H0) and that $R_1(u_0,p_0)$ and $R_2(u_0,p_0)$ exist, we can pass to the limit $\eps\searrow 0$. This finishes the proof. 
\end{proof}

\begin{remark}
    In the next section, we will apply the abstract result of Theorem \ref{thm:diff_lagrange} to the Lagrangian introduced in \eqref{eq:lagrange_scale_eps}. There, it holds that $g(\eps) = J(\Omega_\eps)$ and, when using $\ell(\eps) = |\omega_\eps|$, the derivative \eqref{E:dEll} corresponds to the topological derivative defined in \eqref{def_TD}.
\end{remark}

\section{The topological derivative} \label{sec_topDer}

Let $X=Y= H^1_0(\Dsf)$ and let the Lagrangian $ G$ be defined  as in \eqref{eq:lagrange_scale_eps}. 
We are now going to verify that the hypotheses of Theorem~\ref{thm:diff_lagrange} are satisfied for this $G$ with $\ell(\eps) = |\omega_\eps|$.

\subsection{Analysis of the perturbed state equation}
We introduce the abbreviation $\Ca_\eps(x,y) := \Ca_{\Omega_\eps}(x,y)$ for $x,y \in \VR^d$.
The perturbed state equation reads: find $u_\eps \in H^1_0(\Dsf)$ such that
\ben
\partial_\fp G(\eps, u_\eps, 0)(\varphi)=0 \quad \text{ for all } \varphi \in H^1_0(\Dsf),
\een
or equivalently $ u_\eps \in H^1_0(\Dsf) $ satisfies
\ben\label{eq:state_per}
\int_{\Dsf} \Ca_{\eps}(x, \nabla u_\eps ) \cdot \nabla \varphi \; dx = \int_{\Dsf} f\varphi\;dx \quad \mbox{for all } \varphi \in H^1_0(\Dsf).  
\een
Since \eqref{eq:state_per} admits a unique solution we have that $E(\eps) = \{ u_\eps\}$ is a singleton.  Together with the previous observation that \eqref{eq_adjoint} admits a unique solution, we have that Hypothesis~(H0) is satisfied.
\begin{lemma}\label{lem:u_ueps}
    Let Assumption~\ref{A:nonlinearity}(i),(ii) be satisfied. There is a constant $C>0$, such that for all small $\eps>0$, 
	\ben\label{eq:est_u_eps_D}
    \|u_\eps - u_0\|_{H^1(\Dsf)} \le C\eps^{d/2}. 
	\een
\end{lemma}
\begin{proof}
Subtracting \eqref{eq:state_per} for $\eps >0$ and $\eps=0$ yields
\ben\label{eq:diff_per_state}
\begin{split}
    \int_{\Dsf} &  (\Ca_\eps(x, \nabla u_\eps ) - \Ca_\eps(x, \nabla u_0 ))\cdot \nabla \varphi \; dx \\
                &  = - \int_{\omega_\eps}(a_1 (\nabla u_0 ) - a_2 (\nabla u_0 )) \cdot \nabla \varphi \;dx \quad \text{ for all } \varphi \in H^1_0(\Dsf).
\end{split}
\een
Therefore testing \eqref{eq:diff_per_state} with $\varphi := u_\eps - u_0$, then applying H\"older's inequality and using the monotonicity of $\Ca_\eps$ leads to 
\ben
\|\nabla(u_\eps - u_0)\|_{L_2(\Dsf)^d}^2 \le C \sqrt{|\omega_\eps|} (\|\nabla u_0\|_{C(\overbar B_\delta(z))^d}+1)\|\nabla(u_\eps - u_0)\|_{L_2(\Dsf)^d},
\een
where $0 < \eps <\delta$ and $C$ is a generic constant. Here, we also used \eqref{rem_ai}. Now the result follows from $|\omega_\eps|=|\omega|\eps^d$ and the Poincar\'e inequality.
\end{proof}

\begin{definition}
We define the variation of the state by 
\ben
K_\eps := \frac{(u_\eps - u_0)\circ T_\eps}{\eps} \in H_0^1(\eps^{-1}\Dsf), \quad \eps > 0.
\een
By extending $u_\eps$ and $u_0$ by zero outside of  $\Dsf$, we can view $K_\eps $ as an element of $BL(\VR^d)$ (and its equivalence class $[K_\eps]$ as element of $\dot{BL}(\VR^d)$).
\end{definition}
Our main result of this section is the following theorem:

\begin{theorem}\label{T:Keps_strong_K}
Let Assumption~\ref{A:nonlinearity}(i),(ii) be satisfied.
    \begin{itemize}
\item[(i)] There exists a unique solution $K \in \dot{BL}(\VR^d)$ to
    \ben\label{E:limit_K}
\begin{split}
    \int_{\VR^d} &  (\Ca_\omega(x,  \nabla K + U_0  ) - \Ca_\omega(x, U_0 ))\cdot \nabla \varphi \; dx \\
                 &  = - \int_{\omega}(a_1 (U_0 ) - a_2 (U_0 )) \cdot \nabla \varphi \;dx \quad \text{ for all } \varphi \in  BL(\VR^d),
\end{split}
\een
where $U_0:= \nabla u_0(z)$ and $\Ca_\omega(x,y):= a_1(y)\chi_\omega(x) + a_2(y)\chi_{\VR^d\setminus \omega}(x)$.  
\item[(ii)] We have $\nabla K_\eps \to \nabla K$ strongly in $L_2(\VR^d)^d$ as $\eps  \searrow 0$.
\end{itemize}
\end{theorem}

    Proof of (i): Thanks to Assumption~\ref{A:nonlinearity} the operator $B_\omega: \dot{BL}(\VR^d) \to \dot{BL}(\VR^d)^*$ defined by $\langle B_\omega\varphi,\psi\rangle := \int_{\VR^d} (\Ca_\omega(x,  \nabla \varphi + U_0  ) - \Ca_\omega(x, U_0 ))\cdot \nabla \psi \; dx$ is a strongly monotone and Lipschitz continuous and hence the unique solvability follows by the theorem of Zarantonello; see \cite[p.504, Thm. 25.B]{b_ZE_1990a}. 
    
Proof of (ii): We split the proof into two lemmas. The idea is as follows:
\begin{enumerate}
    \item introduce the intermediate quantity $H_\eps$ and split $K - K_\eps = K-H_\eps + H_\eps - K_\eps$,
    \item show $K-H_\eps \to 0$,
    \item show  $H_\eps - K_\eps \to 0$.
\end{enumerate}
This splitting is not necessary, but simplifies the presentation.  Note that changing 
variables in \eqref{eq:est_u_eps_D} gives
\ben
\|\nabla K_\eps \|_{L_2(\VR^d)}\le C \quad \text{ for all } \eps > 0.
\een
We start by changing variables in \eqref{eq:diff_per_state} to obtain an equation for $K_\eps$:
\ben\label{E:per_eps}
\begin{split}
    \int_{\VR^d} &  (\Ca_\omega(x,\nabla K_\eps + \nabla u_0(x_\eps) ) - \Ca_\omega(x, \nabla u_0(x_\eps) ))\cdot \nabla \varphi \; dx \\
                &  = - \int_{\omega}(a_1 (\nabla u_0 (x_\eps) ) - a_2 (\nabla u_0(x_\eps) )) \cdot \nabla \varphi \;dx
\end{split}
\een
for all $\varphi \in H^1_0(\eps^{-1}\Dsf)$  where we recall the notation $x_\eps = T_\eps(x)= \eps x$. Similarly as in \cite{a_AMBO_2017a, AmstutzGangl2019} we approximate $K_\eps$ by $H_\eps \in H_0^1(\eps^{-1}\Dsf)$ solution to 
\ben\label{E:approx_Keps}
\begin{split}
\int_{\VR^d}  & (\Ca_\omega(x,\nabla H_\eps + U_0 ) - \Ca_\omega(x, U_0 ))\cdot \nabla \varphi \; dx\\
              & = - \int_{\omega}(a_1 (U_0) - a_2 (U_0)) \cdot \nabla \varphi \;dx \quad \text{ for all } \varphi \in H^1_0(\eps^{-1} \Dsf).
\end{split}
\een
This equation is simply \eqref{E:per_eps} with $\nabla u(x_\eps)$ replaced by $U_0$. 
 We now introduce the projection of $K$ into the space $H^1_0(\eps^{-1}\Dsf)$:  For this we consider the more general situation of $\dot{BL}_p(\VR^d)$. 
    \begin{definition}\label{D:projection}
Let $\eps > 0$ and $1< p <\infty$.   For every $K\in \dot{BL}_p(\VR^d)$ we define its projection $P_\eps(K)\in W^1_{p,0}(\VR^d)$ as the minimiser of 
    \ben\label{E:hat_K_eps}
\min_{\varphi \in W^1_{p,0}(\eps^{-1}\Dsf)} \| \nabla(\varphi - K)\|_{L_p(\eps^{-1}\Dsf)^d}. 
    \een
    So $P_\eps: \dot{BL}_p(\VR^d) \to W^1_{p,0}(\eps^{-1}\Dsf) \subset W^1_{p,0}(\VR^d) $ is a nonlinear operator. 
    \end{definition}
The next lemma shows that the operator $P_\eps$ is continuous with respect to $\eps$.

\begin{lemma}\label{L:projection_convergence}
    For every $K\in \dot{BL}_p(\VR^d)$ it holds that
        \ben\label{E:min_project}
    \nabla (P_\eps(K))  \to \nabla K \quad  \mbox{ strongly in }L_p(\VR^d)^d \text{ as }  \eps \searrow 0.
        \een       
    \end{lemma}
    \begin{proof}
        Since $\varphi \mapsto \|\nabla(\varphi-K)\|_{L_p(\eps^{-1}\Dsf)^d}^p$ is strictly convex on 
    $W^1_{p,0}(\eps^{-1}\Dsf)$ it follows that \eqref{E:hat_K_eps} admits a unique solution which is denoted by $P_\eps(K)$. We have by definition

 \ben\label{eq:inequality}
\| \nabla(P_\eps(K) - K)\|_{L_p(\eps^{-1}\Dsf)^d} \le \| \nabla(\varphi - K)\|_{L_p(\eps^{-1}\Dsf)^d} \quad \text{ for all } \varphi \in W^1_{p,0}(\eps^{-1}\Dsf). 
\een
Choosing a function $\varphi \in W^1_{p,0}(\eps^{-1}\Dsf) $ with fixed support in some compact set $K\subset \Dsf$, we see that we find $C>0$, such that $\|\nabla(P_\eps(K))\|_{L_p(\VR^d)^d}\le C$ for all $\eps \in (0,1)$. Now let $\tilde \eps >0$ be arbitrary. Let $(\eps_n)$ be a null-sequence, such that $\eps_n <\tilde \eps$ for all $n\geq 1$. We obtain from \eqref{eq:inequality} that for all $n\ge 1$:
\ben\label{eq:Peps_tilde}
\| \nabla(P_{\eps_n}(K) - K)\|_{L_p(\eps_n^{-1}\Dsf)^d} \le \| \nabla(\varphi - K)\|_{L_p(\eps_n^{-1}\Dsf)^d} \quad \text{ for all } \varphi \in W^1_{p,0}(\tilde \eps^{-1}\Dsf).
\een
Here we extended $\varphi$ to $\eps_n^{-1}\Dsf$ by zero. Since $P_{\eps_n}(K)$ is bounded in $\dot{BL}_p(\VR^d)$ we find a weakly converging subsequence (denoted the same) and an element $\hat K\in \dot{BL}_p(\VR^d)$, such that\\
$\liminf_{n\to \infty} \|\nabla(P_{\eps_n}(K))\|_{L_p(\VR^d)^d} \ge \|\nabla \hat K\|_{L_p(\VR^d)^d}.$ Hence it follows from \eqref{eq:Peps_tilde} that 
\ben\label{eq:Peps_tilde2}
\| \nabla(\hat K - K)\|_{L_p(\VR^d)^d} \le \| \nabla(\varphi - K)\|_{L_p(\VR^d)^d} \quad \text{ for all } \varphi \in W^1_{p,0}(\tilde \eps^{-1}\Dsf).
\een
But since $\tilde \eps >0$ was arbitrary and since $C^\infty_c(\VR^d)/\VR$ is dense in $\dot{BL}_p(\VR^d)$ it follows 
\ben\label{eq:Peps_tilde3}
\| \nabla(\hat K - K)\|_{L_p(\VR^d)^d} \le \liminf_{n\to \infty} \| \nabla(P_{\eps_n}(K) - K)\|_{L_p(\eps_n^{-1}\Dsf)^d} \le  \| \nabla(\varphi - K)\|_{L_p(\VR^d)^d}
\een
for all $\varphi \in \dot{BL}_p(\VR^d)$. However, by choosing $\varphi = K$, this implies $\hat K = K$. It follows in particular that 
$P_{\eps}(K) \rightharpoonup K$ weakly in $L_p(\VR^d)$ as $\eps \searrow 0$. In addition 
it follows from \eqref{eq:Peps_tilde3} the norm convergence of $\nabla (P_\eps(K))$ in $L_p(\VR^d)^d$. Hence by the theorem of Radon-Riesz (see \cite[p.264, Thm.5.10]{b_EL_1999a}) we have 
$\|\nabla (P_\eps(K) - K) \|_{L_p(\VR^d)^d} \to 0$ as $\eps\searrow 0$. 
\end{proof}
 We now let $\hat K_\eps:= P_\eps(K)\in H^1_0(\eps^{-1}\Dsf)$ be the solution to \eqref{E:hat_K_eps} with $p=2$. 

As for $K_\eps$, we can also view $H_\eps$ and $\hat K_\eps$ as elements of $BL(\VR^d)$ by extending them by $0$ outside $\eps^{-1}\Dsf$.

\begin{remark}

In  \cite{a_AMBO_2017a,AmstutzGangl2019} the proof of $\nabla K_\eps \to \nabla K$ strongly in $L_2(\VR^d)^d$ as $\eps \searrow 0$ was given using a cut-off argument of $K$. The reason is that one cannot directly work with $K$ since $K \not\in H^1_0(\eps^{-1}\Dsf)$ for every $\eps>0$. This cut-off technique lead to technical arguments which required additional smoothness of the operators, some restrictions on the non-linearity and also to restrict to $\omega=B_1(0)$. As we will see by introducing the projection $\hat K_\eps$ this step is simplified substantially.

\end{remark}

\begin{lemma}\label{L:Heps_K}
We have
\ben
\nabla H_\eps \to \nabla K \quad \text{ strongly in } L_2(\VR^d)^d \text{ as }  \eps \searrow 0. 
\een
    \end{lemma}
    \begin{proof}
Subtracting \eqref{E:approx_Keps} from \eqref{E:limit_K} yields after rearranging:
\ben\label{E:rewrite_first_approx}
    \int_{\VR^d}   (\Ca_\omega(x,\nabla \hat K_\eps + U_0 ) - \Ca_\omega(x,\nabla H_\eps + U_0 ))\cdot \nabla \varphi \; dx
                   = \int_{\VR^d}   (\Ca_\omega(x,\nabla \hat K_\eps + U_0 ) - \Ca_\omega(x, \nabla K + U_0 ))\cdot \nabla \varphi \; dx
\een
for all $\varphi \in H^1_0(\eps^{-1} \Dsf)$. Now we test this equation with $\varphi = \hat K_\eps - H_\eps \in H^1_0(\eps^{-1}\Dsf)$, use the monotonicity of $\Ca_\omega$ and H\"older's inequality:
\ben\label{E:strong_Keps_Heps}
\begin{split}
    C\|\nabla(\hat K_\eps - H_\eps) \|_{L_2(\VR^d)^d}^2 &\le 
    \int_{\VR^d}   (\Ca_\omega(x,\nabla \hat K_\eps + U_0 ) - \Ca_\omega(x,\nabla H_\eps + U_0 ))\cdot \nabla (\hat K_\eps - H_\eps) \; dx
                  \\
                                                       & \stackrel{\eqref{E:rewrite_first_approx}}{=} \int_{\VR^d}   (\Ca_\omega(x,\nabla \hat K_\eps + U_0 ) - \Ca_\omega(x, \nabla  K + U_0 ))\cdot \nabla (\hat K_\eps - H_\eps) \; dx\\
                                              & \le\int_{\VR^d} |\nabla (\hat K_\eps - K)||\nabla(\hat K_\eps - H_\eps)|\;dx \\
                                              & \le \|\nabla(\hat K_\eps - K) \|_{L_2(\VR^d)^d} \|\nabla(\hat K_\eps - H_\eps) \|_{L_2(\VR^d)^d}.
\end{split}
\een
Since  in view of Lemma~\ref{L:projection_convergence}, we have $\nabla\hat K_\eps \to \nabla K$ strongly in $L_2(\VR^d)^d$ it follows from \eqref{E:strong_Keps_Heps} that  $\nabla (\hat K_\eps - H_\eps) \to 0$ strongly in $L_2(\VR^d)^d$ and therefore also $\|\nabla(H_\eps - K)\|_{L_2(\VR^d)^d} \le \|\nabla(H_\eps - \hat K_\eps)\|_{L_2(\VR^d)^d} + \|\nabla(\hat K_\eps - K)\|_{L_2(\VR^d)^d} \to 0$ as $\eps \searrow 0$.
    \end{proof}

    We now prove that $\nabla (H_\eps - K_\eps) \to 0$ strongly in $L_2(\VR^d)^d$.
    \begin{lemma}\label{L:Heps_Keps}
We have
\ben
\nabla (H_\eps - K_\eps) \to 0 \quad \text{ strongly in }  L_2(\VR^d)^d \; \text{ as } \eps \searrow 0.
\een
    \end{lemma}
    \begin{proof}
        Subtracting \eqref{E:per_eps} and \eqref{E:approx_Keps} we obtain
\ben
\begin{split}
    \int_{\VR^d} &  (\Ca_\omega(x,\nabla K_\eps + \nabla u_0(x_\eps) ) - \Ca_\omega(x,\nabla H_\eps + U_0 ))\cdot \nabla \varphi \; dx \\
                 & + \int_{\VR^d}   (\Ca_\omega(x,U_0 ) - \Ca_\omega(x,\nabla u_0(x_\eps) ))\cdot \nabla \varphi \; dx \\
                &  = - \int_{\omega}(a_1 (\nabla u_0 (x_\eps) ) - a_2 (\nabla u_0(x_\eps) )) \cdot \nabla \varphi  + (a_1 (U_0 ) - a_2 (U_0 )) \cdot \nabla \varphi \;dx
\end{split}
\een
for all $\varphi \in H^1_0(\eps^{-1} \Dsf)$. In order to be able to use the monotonicity of $\Ca_\omega$ we rewrite this as follows 
\ben\label{E:diff_Heps_Keps}
\begin{split}
    \int_{\VR^d} &  (\Ca_\omega(x,\nabla K_\eps + \nabla u_0(x_\eps) ) - \Ca_\omega(x,\nabla H_\eps + \nabla u_0(x_\eps) ) ))\cdot \nabla \varphi \; dx \\
    = &  \underbrace{  - \int_{\VR^d}  ( (\Ca_\omega(x,\nabla H_\eps + \nabla u_0(x_\eps) ) - (\Ca_\omega(x,\nabla H_\eps + U_0 )\cdot \nabla \varphi \; dx}_{=:I_1(\eps,\varphi)} \\
                 & \underbrace{- \int_{\VR^d}   (\Ca_\omega(x,U_0 ) - \Ca_\omega(x,\nabla u_0(x_\eps) ))\cdot \nabla \varphi \; dx}_{=:I_2(\eps,\varphi)} \\
                 &  \underbrace{ - \int_{\omega}(a_1 (\nabla u_0 (x_\eps) ) - a_2 (\nabla u_0(x_\eps) )) \cdot \nabla \varphi \;dx  + (a_1 (U_0 ) - a_2 (U_0 )) \cdot \nabla \varphi \;dx}_{=:I_3(\eps,\varphi)}.
\end{split}
\een
Since $a_i$ are Lipschitz continuous and $u_0 \in C^{1,\alpha}(\overbar{B_\delta(z)})$ with $\alpha,\delta>0$, we immediately obtain that $|I_3(\eps,\varphi)| \le C\eps^\alpha \|\nabla \varphi\|_{L_2(\VR^d)^d} $ for a suitable constant $C>0$. We now show that also $|I_1(\eps,\varphi)+I_2(\eps,\varphi)|\le C(\eps)\|\nabla \varphi\|_{L_2(\VR^d)^d}$ and $C(\eps)\to 0$ as $\eps \searrow 0$. We write for  arbitrary $r\in (0,1$),
\ben\label{E:sum_I1_I2}
\begin{split}
    I_1(\eps,\varphi)+I_2(\eps,\varphi)  = &  - \int_{B_{\eps^{-r}}}  ( (\Ca_\omega(x,\nabla H_\eps + \nabla u_0(x_\eps) ) - (\Ca_\omega(x,\nabla H_\eps + U_0 )\cdot \nabla \varphi \; dx \\
                                           & - \int_{B_{\eps^{-r}}}   (\Ca_\omega(x,U_0 ) - \Ca_\omega(x,\nabla u_0(x_\eps) ))\cdot \nabla \varphi \; dx \\
                                           &  - \int_{\VR^d\setminus B_{\eps^{-r}}}  ( (\Ca_\omega(x,\nabla H_\eps + \nabla u_0(x_\eps) ) - (\Ca_\omega(x, \nabla u_0(x_\eps) )\cdot \nabla \varphi \; dx \\
                                           &  + \int_{\VR^d\setminus B_{\eps^{-r}}}   ((\Ca_\omega(x,\nabla H_\eps + U_0 ) - \Ca_\omega(x, U_0 ))\cdot \nabla \varphi \; dx.
\end{split}
\een
As in \cite[Prop. 6.7]{a_AMBO_2017a} the idea of choosing a power $\eps^{-r}$ is to let the ball $B_{\eps^{-r}}(0)$ expand slower than $B_{\eps^{-1}}(0)$ by choosing $r\in (0,1)$ appropriately.
Now we can estimate the right hand side of \eqref{E:sum_I1_I2} using the Lipschitz continuity of $a_i$ (see Assumption~\ref{A:nonlinearity}(ii)) as follows
\ben
\begin{split}
    |I_1(\eps,\varphi) + I_2(\eps,\varphi)| & \le  2C\int_{B_{\eps^{-r}}} |U_0-\nabla u_0(x_\eps)||\nabla \varphi| \;dx   
                                                   + 2C\int_{\VR^d\setminus B_{\eps^{-r}}}|\nabla H_\eps ||\nabla \varphi|\;dx \\
                                        &\leq C \int_{B_{\eps^{-r}}} |x_\eps|^\alpha |\nabla \varphi| \mbox dx           + 2C\int_{\VR^d\setminus B_{\eps^{-r}}}|\nabla H_\eps ||\nabla \varphi|\;dx \\                                        
                                                 & \le \eps^{-r \alpha} \eps^\alpha \eps^{-rd/2} C  \|\nabla \varphi\|_{L_2(\VR^d)^d} + 2C \|\nabla H_\eps\|_{L_2(\VR^d\setminus B_{\eps^{-r}})^d}\|\nabla \varphi\|_{L_2(\VR^d\setminus B_{\eps^r})^d}
\end{split}
\een
For $r$ sufficiently close to $0$, we have $\eps^{-r\alpha} \eps^\alpha \eps^{-rd/2} = \eps^{\alpha - r(\frac{d}{2} +\alpha)} \to 0$. Moreover, by the triangle inequality we have 
\ben
\|\nabla H_\eps \|_{L_2(\VR^d\setminus B_{\eps^{-r}})} \le \|\nabla (H_\eps- K) \|_{L_2(\VR^d\setminus B_{\eps^{-r}})} + \|\nabla K \|_{L_2(\VR^d\setminus B_{\eps^{-r}})}.
\een
The first term on the right hand side goes to zero in view of  Lemma~\ref{L:Heps_K}. The second
term goes to zero since $\nabla K\in L_2(\VR^d)^d$ thus $\|\nabla K\|_{L_2(\VR^d\setminus B_{\eps^{-r}})^d} \to 0$ as $\eps \searrow 0$. Using $K_\eps - H_\eps$ as test function in \eqref{E:diff_Heps_Keps}, using the monotonicity of $\Ca$ and employing $|I_1(\eps,\varphi)+I_2(\eps,\varphi)+I_3(\eps,\varphi)| \le C(\eps)\|\nabla \varphi\|_{L_2(\VR^d)^d}$ with $C(\eps) \to 0$ as $\eps \searrow 0$, shows the result.
    \end{proof}
    Combining Lemma~\ref{L:Heps_K} and Lemma~\ref{L:Heps_Keps} proves Theorem~\ref{T:Keps_strong_K}(ii). \hfill $square$

    We get the following properties of the sequence $(\eps K_\eps)$:
    \begin{corollary}\label{cor:epsK}
        We have
        \ben
 \eps K_\eps \to 0  \quad \left\{ \begin{array}{lll}
         \text{ strongly in } L_p(\VR^d) &\quad \text{ for } d=2, &\; p\in (2,4],\\
         \text{ strongly in } L_{p}(\VR^d) &\quad \text{ for } d\ge 3, &\; p\in (2,2^*],\\
         \text{ weakly in } L_2(\VR^d) & \quad \text{ for } d\ge 2, &
 \end{array} \right. 
        \een
        where $2^*:= 2d/(d-2)$ denotes the Sobolev exponent of $2$ for $d\ge 3$. 
    \end{corollary}
    \begin{proof}
        Let $d=2$. From the Ladyzhenskaya inequality (see \cite{a_LA_1959a}) we obtain the estimate $\|\eps K_\eps\|_{L_4(\VR^2)} \le C\eps^{1/2} \|\eps K_\eps\|_{L_2(\VR^2)}^{1/2}\|\nabla K_\eps\|_{L_2(\VR^2)^2}^{1/2}$. Hence for $d=2$, we conclude $\eps K_\eps \to 0$ in $L_4(\VR^2)$ as  $\eps\searrow 0$. 
        Let now $p\in(2,4)$. Then in view of the interpolation inequality $\| \eps K_\eps\|_{L_p(\VR^2)} \le \|\eps K_\eps \|_{L_2(\VR^2)}^\theta \|\eps K_\eps\|_{L_4(\VR^2)}^{1-\theta}$ for all $\theta\in (0,1)$ and $\frac{1}{p} = \frac{\theta}{2} + \frac{(1-\theta)}{4}$ it follows $\eps K_\eps \to 0$ in $L_p(\VR^2)$ as $\eps \searrow 0$. 

    Let now $d\ge 3$. By the Gagliardo-Nirenberg inequality (see \cite{a_NI_1959a}) we obtain $\|\eps K_\eps \|_{L_{2*}(\VR^d)} \le C\eps  \|\nabla K_\eps \|_{L_2(\VR^d)^d}$ and hence  $\eps K_\eps \to 0$ strongly in $L_{2^*}(\VR^d)$ as $\eps \searrow 0$. The convergence $\eps K_\eps\to 0$ strongly in  $L_p(\VR^d)$ as $\eps\searrow 0$ for all $p\in (2,2^*)$ follows by the interpolation argument as in the previous step.

    The convergence $\eps K_\eps \tow 0$ weakly in $L_2(\VR^d)$ as $\eps \searrow 0$ can be proved using the same arguments as in \cite[Thm. 4.14]{Sturm2019}.

\end{proof}

\subsection{Computation of $R_1(u_0,p_0)$ and $R_2(u_0,p_0)$}

It is easily seen from the continuity of $a_1$, $a_2$, $\nabla u_0$ and $\nabla p_0$ that $G$ is $\ell$-differentiable with
\ben
    \partial_\ell G(0,u_0,p_0) = \frac{1}{|\omega|} (a_1 (U_0 ) - a_2 (U_0 )) \cdot \int_\omega P_0 \, dx .
\een

It remains to check that the limits of $R_1(u_0,p_0)$ and $R_2(u_0,p_0)$ exist. For this we use Assumption~\ref{A:nonlinearity}(i)-(iii). Using the change of variables $T_\eps$, we have
\ben\label{eq:R1_eps}
\begin{split}
    R_1^\eps (u_0,p_0) & =\frac{1}{\ell(\eps)} \int_0^1 \int_{\Dsf} \left(\partial_u \Ca_{\eps}(x, \nabla (su_\eps  + (1-s)u_0)) - \partial_u \Ca_{\eps}(x, \nabla u_0 ) \right)(\nabla (u_\eps-u_0))  \cdot \nabla p_0\;dx\;  ds \\
                       &  + \frac{1}{\ell(\eps)} \int_{\Dsf} |\nabla (u_\eps  - u_0)|^2 \; dx\\
                   & =  \frac{1}{|\omega|} \int_0^1 \int_{\VR^d} \left(\partial_u \Ca_{\omega}(x,  s\nabla K_\eps   + \nabla u_0(x_\eps) ) - \partial_u \Ca_{\omega}(x, \nabla u_0(x_\eps) ) \right)(\nabla K_\eps)  \cdot \nabla p_0 (x_\eps)\;dx\;  ds \\
                   & + \frac{1}{|\omega|}\int_{\VR^d} |\nabla K_\eps|^2 \; dx  \\
                   & \to \frac{1}{|\omega|} \int_0^1 \int_{\VR^d} \left(\partial_u \Ca_{\omega}(x,  s\nabla K   + U_0) - \partial_u \Ca_{\omega}(x, U_0 ) \right)(\nabla K )  \cdot P_0\;dx\;  ds + \frac{1}{|\omega|}\int_{\VR^d} |\nabla K|^2 \; dx.
\end{split}
\een
Here, we used that $\nabla K_\eps \to \nabla K$ strongly in $L_2(\VR^d)^d$ as $\eps\searrow 0$ for the limit of the second term. To see the convergence of the first term, we may write
\begin{align*}
   \int_0^1 \int_{\VR^d} (\partial_u & \Ca_{\omega}(x,  s\nabla K_\eps    + \nabla u_0(x_\eps) ) - \partial_u \Ca_{\omega}(x, \nabla u_0(x_\eps) ) )(\nabla K_\eps)  \cdot \nabla p_0 (x_\eps)\;dx ds = \\
                &+ \int_0^1 \int_{\VR^d}(\partial_u \Ca_{\omega}(x,  s\nabla K_\eps   + \nabla u_0(x_\eps) ) - \partial_u \Ca_{\omega}(x,  s\nabla K   + \nabla u_0(x_\eps) ) )(\nabla K_\eps)  \cdot \nabla p_0 (x_\eps) \;dx ds \\
                &+ \int_0^1 \int_{\VR^d}(\partial_u \Ca_{\omega}(x,  s\nabla K   + \nabla u_0(x_\eps) ) - \partial_u \Ca_{\omega}(x, \nabla u_0(x_\eps) ) )(\nabla (K_\eps-K))  \cdot \nabla p_0 (x_\eps) \;dx ds\\
                                                          & +\int_0^1 \int_{\VR^d}(\partial_u \Ca_{\omega}(x,  s\nabla K   + \nabla u_0(x_\eps) ) - \partial_u \Ca_{\omega}(x, \nabla u_0(x_\eps) ) )(\nabla K)  \cdot \nabla p_0 (x_\eps)\;dx ds.
\end{align*}
Using Assumption~\ref{A:nonlinearity}(iii) and $\nabla p_0 \in L^\infty(\Dsf)^d$, we see that the 
absolute value of the first and second term on the right hand side can be bounded by $C\|\nabla(K_\eps-K)\|_{L_2(\VR^d)^d}\|\nabla K\|_{L_2(\VR^d)^d}$  and $C\|\nabla(K_\eps-K)\|_{L_2(\VR^d)^d}\|\nabla K_\eps\|_{L_2(\VR^d)^d}$, respectively,   and hence using $\nabla K_\eps \to \nabla K$ in $L_2(\VR^d)^d$ as $\eps\searrow 0$ they disappear in the limit. The last term converges to the desired limit by using Lebesgue's dominated convergence theorem. 
Using the fundamental theorem, we obtain the expression in \eqref{E:R_term}. Similarly, using \eqref{rem_aiii}, the continuity of $\nabla u_0$ and $\nabla p_0$ at $z$, the continuity of $\partial_ua_1,\partial_ua_2$, and again $\nabla K_\eps \to \nabla K$ strongly in $L_2(\VR^d)^d$, we obtain by Lebesgue's dominated convergence theorem
\ben
\begin{split}
    R_2^\eps (u,p) & = \frac{1}{\ell(\eps)} \int_{\omega_\eps } (\partial_u a_1(\nabla u_0) - \partial_u  a_2 (\nabla u_0))(\nabla (u_\eps-u_0))  \cdot \nabla p_0\;dx\\
                   & = \frac{1}{|\omega|}\int_{\omega} (\partial_u a_1(\nabla u_0(x_\eps)) - \partial_u a_2(\nabla u_0(x_\eps)))(\nabla K_\eps)  \cdot \nabla p_0(x_\eps)\;dx \\
                   & \to \frac{1}{|\omega|}\int_{\omega} (\partial_u a_1(U_0) - \partial_u a_2(U_0))(\nabla K)  \cdot P_0\;dx.
\end{split}
\een
This finishes the proof of the Main Theorem.

\begin{remark}
    We remark that, while the problem considered in this paper is in an $L_2$ setting, the projection trick of Definition \ref{D:projection} is also possible in $W^1_{p}$ spaces. Thus, an extension of our results for the topological derivative to PDE constraints posed in an $L_p$ setting as considered in \cite{a_AMBO_2017a} with $1 < p < \infty$, $p \neq 2$ seems possible.
\end{remark}

\newcommand{\varAd}{ \tilde Q }
\begin{remark}
    The obtained formula for the topological derivative coincides with the formulas obtained in \cite[Thm. 4.4]{a_AMBO_2017a} and \cite[Thm. 2 and Thm. 3]{AmstutzGangl2019} for the respective special cases, which can be seen as follows: Introducing the problem defining the variation of the adjoint state $\varAd \in \dot{BL}(\VR^d)$,
    \ben \label{eq_varAdjoint}
        \int_{\VR^d} \partial_u \Ca_\omega(x,U_0)(\nabla \varphi) \cdot \nabla \varAd \,dx = -\int_\omega (\partial_u a_1(U_0) - \partial_u a_2(U_0))(\nabla \varphi) \cdot P_0\; dx
    \een
    for all $\varphi \in BL(\VR^d)$, and adding the left and right hand side of \eqref{E:limit_K} tested with the solution $\varAd$ of \eqref{eq_varAdjoint}, the term $R_2(u_0, p_0)$ can be rewritten as
    \ben
        \begin{split}
            R_2(u_0,p_0) =& - \frac{1}{|\omega|} \int_{\VR^d} \partial_u \Ca_\omega(x,U_0)(\nabla K) \cdot \nabla \varAd \,dx \\
            =&\frac{1}{|\omega|} \int_{\VR^d} (\Ca_\omega(x, \nabla K  +U_0) - \Ca_\omega(x, U_0 ) -\partial_u \Ca_\omega(x,U_0)(\nabla K) )\cdot \nabla \varAd \; dx \\
                 &+ \frac{1}{|\omega|}\int_{\omega}(a_1 (U_0 ) - a_2 (U_0 )) \cdot \nabla \varAd \;dx.
        \end{split}
    \een
    Together with the terms $\partial_\ell G(0,u_0,p_0)$ and $R_1(u_0, p_0)$, the topological derivative reads
    \ben \label{eq_TDotherFormat}
    \begin{split}
        dJ(\Omega)(z) = \frac{1}{|\omega|} &\left[ (a_1 (U_0 ) - a_2 (U_0 )) \cdot \int_\omega P_0+\nabla \varAd \, dx \right. \\
        & \left. + \int_{\VR^d}(\Ca_\omega(x, \nabla K  +U_0) - \Ca_\omega(x, U_0 ) -\partial_u \Ca_\omega(x,U_0)(\nabla K) )\cdot (P_0 + \nabla \varAd) \, dx \right. \\
        & \left. + \int_{\VR^d} |\nabla K|^2 \, dx \right]
    \end{split}
    \een
    which is, up to a scaling by $1/|\omega|$ the same formula as obtained in \cite{a_AMBO_2017a} and \cite{AmstutzGangl2019}. The different scaling is due to a different definition of the topological derivative in these publications.
\end{remark}

\begin{remark}
    It can be seen from \eqref{eq_varAdjoint} that $\nabla \varAd$ depends linearly on $P_0$. Thus, it can be shown that there exists a matrix $\mathcal M = \mathcal M(\omega,\partial_u a_1(U_0), \partial_u a_2(U_0))$, which is related to the concept of polarization matrices \cite{AmmariKang2007}, such that $\int_\omega \nabla \varAd \, dx = \mathcal M P_0$, see also \cite[Sec. 6]{AmstutzGangl2019} for the special setting of two-dimensional magnetostatics. 
    
    For a discussion on the efficient numerical evaluation of the second integral in \eqref{eq_TDotherFormat} involving $K$, see \cite[Sec. 7]{AmstutzGangl2019}.
\end{remark}

\section{Comparison with the averaged adjoint approach and more general cost functions}\label{sec_averaged_adjoint}
In this section we compare the Lagrangian approach of the previous section with the averaged adjoint approach; see \cite{Sturm2019}. We demonstrate that the averaged adjoint approach has some advantages at the price of being more technically involved. In fact, with the averaged adjoint approach we are 
able to treat a cost function of the type:
\ben
J(\Omega):= a \int_{\Dsf}(u-u_d)^2\;dx + b \int_{\Dsf}|\nabla(u-u_d)|^2\;dx
\een
with $a,b\ge 0$ and $u$ the solution to \eqref{eq:state_per} with $\eps=0$ for $\Omega\subset \Dsf$. It can be checked that the first term cannot be directly be handled with the Lagrangian technique of Section~\ref{sec_adjFramework}. In fact in order to pass to the limit $\eps\searrow0$ in \eqref{eq:R1_eps}, we would need $\eps K_\eps \to 0$ strongly in $L_2(\VR^d)^d$, which does not directly follow (see Corollary~\ref{cor:epsK}). Note that this term is not covered by the analysis in \cite{a_AMBO_2017a}.

\subsection{Averaged adjoint}
We use the same setting as in Section~\ref{sec_adjFramework} and let $G$ be a Lagrangian (see Definition~\ref{def:lagrangian}) defined on $[0,\tau]\times X\times Y$ with $X,Y$ being vector spaces.

The key ingredient of the averaged adjoint approach is the averaged adjoint equation. In addition to $G$ as in Definition~\ref{def:lagrangian} we assume that $G$ satisfies:
\begin{assumption*}[H1]\label{ass:b}
    For all $t\in[0,\tau]$, $\varphi,\tilde\varphi,p\in X$ and $\psi\in Y$ the derivative 
    $[0,1] \to \VR:\; s\mapsto \partial_\varphi G(t,\varphi+s\tilde \varphi,\psi)(p)$ is well-defined and integrable 
    on $[0,1]$. 
\end{assumption*}

For a Lagrangian satisfying the previous assumption we can introduce the averaged adjoint equation.

\begin{definition}
  Given $\eps\in [0,\tau]$ and $(\fu_0,\fu_\eps)\in E(0)\times E(\eps)$, the 
	\emph{averaged adjoint state equation} is defined as follows: find $\fp_\eps \in X$, such that
	\ben\label{eq:aa_equation}
	\int_0^1 \!\!\partial_\fu G(\eps,s\fu_\eps + (1-s)\fu_0 ,\fp_\eps)(\varphi)\, ds=0 \quad \text{ for all } \varphi\in X. 
	\een
	For every triplet $(\eps, u_0,u_\eps)$ the set of solutions of \eqref{eq:aa_equation} is denoted by $Y(\eps,\fu_0,\fu_\eps)$ and its elements are referred to as \emph{adjoint states} for $\eps =0$ and \emph{averaged adjoint states} for $\eps >0$. 
\end{definition}

By construction of the averaged adjoint equation we have for $\eps \in [0,\tau]$, 
\ben
G(\eps,\fu_\eps,\fp_\eps)  = G(\eps,\fu_0,\fp_\eps).
\een

The following is an alternative to Theorem~\ref{thm:diff_lagrange} (see \cite{Sturm2019}).

\begin{theorem} \label{thm:diff_lagrange2}
    Let $G$ be an $\ell$-differentiable Lagrangian function satisfying Assumption~(H1). Assume further that for all $\eps \in [0, \tau]$
  \begin{itemize}
    \setlength{\itemsep}{3pt}
\item[(i)] the set $E(\eps) = \{ u_\eps \}$ is a singleton,
\item[(ii)] for $u_0 \in E(0)$, $u_\eps \in E(\eps)$ the set of averaged adjoint states $Y(\eps, u_0, u_\eps )=\{p_\eps\}$ is a singleton,
\item[(iii)]
     the limit  
        \ben
        R(u_0,p_0) := \lim_{\eps\searrow 0} \frac{G(\eps,\fu_0,\fp_\eps)-G(\eps,\fu_0,\fp_0)}{\ell(\eps)} \quad \text{ exists}. 
	\een
\end{itemize}    
Then we have 
\ben
d_\ell g(0) = \partial_\ell G(0,\fu_0,\fp_0) + R(u_0,p_0).
\een
\end{theorem}

\subsection{Analysis of the averaged adjoint equation}

We now apply Theorem~\ref{thm:diff_lagrange2} to $X=Y=H^1_0(\Dsf)$ with Lagrangian $G$ given by
\ben
G(\eps, u, p) :=  a\int_{\Dsf} (u - u_d)^2 \; dx +  b\int_{\Dsf} | \nabla(u - u_d)|^2 \; dx +  \int_{\Dsf} \Ca_{\Omega_\eps}(x, \nabla u) \cdot \nabla p \; dx - \int_{\Dsf} f p \;dx.
\een
We use the same setting as in Subsection~\ref{subsec_main_results}. The averaged adjoint $p_\eps \in H_0^1(\Dsf)$ is defined by
\ben
\int_0^1 \partial_{u}G(\eps, su_\eps + (1-s)u_0, p_\eps)(\varphi)\;ds =0 \quad \text{ for all } \varphi\in H_0^1(\Dsf).
\een
This is equivalent to
\ben\label{eq:adjoint_pert}
\begin{split}
    \int_0^1 \int_{\Dsf} \partial_u \Ca_{\eps}(x, \nabla (su_\eps & + (1-s)u_0))(\nabla \varphi)  \cdot \nabla p_\eps\;dx\;  ds
	       \\
                                                                  &= - \int_{\Dsf} (u_\eps + u_0 - 2u_d )\varphi\; dx- \int_{\Dsf} \nabla (u_\eps + u_0 - 2u_d )\cdot \nabla\varphi\; dx
\end{split}
\een
for all $\varphi \in H^1_0(\Dsf)$. As noted earlier, Problem \eqref{eq:state_per} admits a unique solution under Assumption~\ref{A:nonlinearity}. Moreover, problem \eqref{eq:adjoint_pert} has a unique solution due to Assumption \ref{A:nonlinearity} and Lax-Milgram. Therefore, the assumptions (i) and (ii) of Theorem \ref{thm:diff_lagrange2} are satisfied. The $\ell$-differentiability of $G$ follows again as in the previous section. It remains to show the existence of the limit term $R(u_0, p_0)$. 

The following analysis is similar to the study of the perturbation of the 
state equation. Since the model problem is quasi-linear it is crucial that we have the strong convergence of $\nabla K_\eps \to \nabla K$. 

\begin{lemma}\label{L:bound_p_eps}
There is a constant $C>0$, such that
\ben
\|p_\eps - p_0\|_{H^1(\Dsf)} \le C (\eps^{d/2} + \|u_\eps - u_0\|_{H^1(\Dsf)})  \quad \text{ for all } \eps >0.
\een
\end{lemma}
\begin{proof}
Using \eqref{eq:adjoint_pert} for $\eps >0$ and $\eps =0$ we obtain
\ben\label{E:aver_ad_per}
\begin{split}
    \int_0^1 \int_{\Dsf} \partial_u \Ca_\eps(x, \nabla (su_\eps & + (1-s)u_0))(\nabla \varphi)  \cdot \nabla (p_\eps - p_0)\;dx\;  ds  \\
    + \int_0^1  \int_{\Dsf} \bigg(\partial_u \Ca_\eps(x, \nabla (su_\eps & + (1-s)u_0))- \partial_u \Ca_\eps(x, \nabla u_0 ))(\nabla \varphi)\bigg)  \cdot \nabla p_0\;dx\;  ds\\
    + \int_{\Dsf} \bigg(\partial_u \Ca_\eps(x, \nabla u_0 ) &- \partial_u \Ca_0(x, \nabla u_0)(\nabla \varphi)\bigg)  \cdot \nabla p_0\;dx\\
                                                                                  &+  a  \int_\Dsf (u_\eps - u_0)\varphi \; dx +  b   \int_\Dsf \nabla (u_\eps - u_0)\cdot \nabla \varphi \; dx = 0                                                                                
\end{split}
\een
for all $\varphi \in H_0^1(\Dsf)$. Testing with $\varphi = p_\eps-p_0$, using the boundedness of $\nabla p_0$, H\"older's inequality from Assumption~\ref{A:nonlinearity}  gives the result.
\end{proof}

\begin{definition}
We consider again the variation of the adjoint state
\ben
Q_\eps := \frac{(p_\eps - p_0) \circ T_\eps}{\eps} \in H_0^1(\eps^{-1} \Dsf),\; \eps >0.
\een
\end{definition}
Note that Lemma~\ref{L:bound_p_eps} together with Lemma \ref{lem:u_ueps} implies that 
\ben
\int_{\VR^d} (\eps Q_\eps)^2 + |\nabla Q_\eps|^2 \;dx \le C \quad \text{ for all } \eps >0.
\een
This means that $(Q_\eps )$ is bounded in the Beppo-Levi space $BL(\VR^d)$. We now show the weak 
convergence $Q_\eps\rightharpoonup Q$ in $BL(\VR^d)$ to some $Q\in BL(\VR^d)$. 
It can also be shown that $\eps Q_\eps \to 0$ in $L_2(\VR^d)$.

\begin{theorem}
    We have 
    \ben
    \nabla Q_\eps \tow \nabla Q \quad \text{ weakly in } L_2(\VR^d)^d  \text{ as }\eps \searrow 0,
    \een
    where $Q\in BL(\VR^d)$ is the unique solution to
    \ben\label{E:solution_Q}
    \begin{split}
        \int_{\VR^d}& \int_0^1 \partial_{u} \Ca_\omega(x, s\nabla K + U_0)(\nabla \psi) \cdot \nabla Q \; ds\; dx = \\
                    & - \int_{\VR^d} \int_0^1 \big(\partial_{u} \Ca_\omega(x, s\nabla K + U_0)(\nabla \psi)- \partial_{u} \Ca_\omega(x,U_0)(\nabla \psi)\big)\cdot P_0\;dx \\
         & - \int_\omega (\partial_{u} a_1(U_0) -  \partial_{u} a_2(U_0))(\nabla \psi)\cdot P_0\;dx 
          - b \int_{\VR^d} \nabla K \cdot \nabla \psi \; dx 
         \quad \text{ for all } \psi\in BL(\VR^d),
\end{split}
    \een
with $P_0 := \nabla p_0(z)$ and $K$ defined in \eqref{E:limit_K}.
\end{theorem}
\begin{proof}
Changing variables in \eqref{E:aver_ad_per} and rearranging yields
\ben\label{E:aver_ad_per_rescale}
\begin{split}
    \int_0^1 \int_{\VR^d}  \partial_u \Ca_\omega(x, s \nabla K_\eps & + \nabla u_0(x_\eps))  (\nabla \psi)  \cdot \nabla Q_\eps \;dx\;  ds = \\
    - \int_0^1  \int_{\VR^d} \bigg(\partial_u \Ca_\omega(x, s\nabla K_\eps & + \nabla u_0(x_\eps)))- \partial_u \Ca_\omega(x, \nabla u_0(x_\eps) ))(\nabla \psi)\bigg)  \cdot \nabla p_0(x_\eps)\;dx\;  ds\\
    -\int_{\omega} \bigg(\partial_{u} a_1(\nabla u_0(x_\eps)) &- \partial_{u} a_2(\nabla u_0(x_\eps))(\nabla \psi)\bigg)  \cdot \nabla p_0\;dx\;  ds \\
                                                              &  -  a\int_{\VR^d} \eps K_\eps \psi \, dx -  b\int_{\VR^d} \nabla K_\eps \cdot \nabla\psi \, dx=0                                                           
\end{split}
\een
for all $\psi\in H^1_0(\eps^{-1}\Dsf)$. Using $\nabla K_\eps \to \nabla K$ strongly in $L_2(\VR^d)^d$  and $\eps K_\eps \rightharpoonup 0$ weakly in $L_2(\VR^d)$, we can use the Lebesgue dominated convergence theorem pass to the limit in \eqref{E:aver_ad_per_rescale} (for a subsequence) and obtain that the weak limit of the subsequence of $(Q_\eps)$ satisfies \eqref{E:solution_Q}. Since the solution to \eqref{E:solution_Q} is unique we conclude that $ Q_\eps \tow Q$ weakly in $BL(\VR^d)$. 
\end{proof}

\subsection{Computation of $R( u_0, p_0)$}
\begin{lemma}
We have
\ben\label{eq:R_term}
R( u_0, p_0) = (a_1 (U_0 ) - a_2 (U_0) ) \cdot \fint_{\omega} \nabla Q \; dx,
\een
where $Q$ is the solution to \eqref{E:solution_Q}.
\end{lemma}
\begin{proof}
Testing \eqref{eq:state_per} for $\eps =0 $ with $\varphi :=  p_\eps -  p_0$ yields
\ben\label{eq:state_u0}
\int_{\Dsf} \Ca_0 (x, \nabla u_0 )  \cdot \nabla ( p_\eps -  p_0)  \; dx = \int_{\Dsf} f( p_\eps -  p_0) \;dx. 
\een
Therefore
\ben\label{eq:lagrange_difference}
\begin{split}
  G(\eps, u_0, p_\eps) - G(\eps, u_0, p_0) & = \int_{\Dsf}\Ca_\eps (x, \nabla u_0 ) \cdot \nabla ( p_\eps -  p_0) \; dx - \int_{\Dsf} f( p_\eps -  p_0) \;dx\\
                                     & \stackrel{\eqref{eq:state_u0}}{=} \int_{\Dsf}\left(\Ca_\eps (x, \nabla u_0 ) - \Ca_0 (x, \nabla u_0 ) \right)\cdot \nabla ( p_\eps -  p_0) \; dx \\
					 & = \int_{\omega_\eps} (a_1 (\nabla u_0 ) - a_2 (\nabla u_0 ) )   \cdot \nabla ( p_\eps -  p_0) \; dx.
\end{split}
\een
Therefore invoking the change of variables $T_\eps$ in \eqref{eq:lagrange_difference} leads to 
\ben
\frac{G(\eps, u, p_\eps) - G(\eps, u, p)}{|\omega_\eps|} = \frac{1}{|\omega|}\int_{\omega} (a_1 (\nabla u_0(x_\eps)  ) - a_2 (\nabla u_0(x_\eps) ) )  \cdot \nabla Q_\eps \; dx.
\een
In view of the continuity of $a_1,a_2, \nabla  u$ and the weak convergence $\nabla Q_\eps \tow \nabla Q$ in $L_2(\VR^d)^d$, we see that the right hand side converges to the expression \eqref{eq:R_term}.
\end{proof}

\subsection{The final expression of the topological expansion}
So we see that all conditions of Theorem~\ref{thm:diff_lagrange2} are satisfied and we 
have
\ben
dJ(\Omega)(z) = \partial_{\ell} G(0, u_0, p_0) + R( u_0, p_0),
\een
with $R(u_0, p_0)$ given by \eqref{eq:R_term}. We see that the second term on the right hand side 
still depends on $Q$, which we can express through $u_0$ and $p_0$ as follows. First 
we test \eqref{E:solution_Q} with $\psi := K$ and use the fundamental theorem to obtain  
 \begin{equation}
     \begin{split}
  &\int_{ \VR^d }  \left(\Ca_\omega(x, \nabla K +U_0 )-\Ca_\omega(x, U_0 ))\right)\cdot \nabla Q\;dx =   \\
  & - \int_{\VR^d} \int_0^1 \big(\partial_{u} \Ca_\omega(x, s\nabla K + U_0)(\nabla K)- \partial_{u} \Ca_\omega(x,U_0)(\nabla K)\big)\cdot P_0\;ds \;dx \\
         & - \int_\omega (\partial_{u} a_1(U_0) -  \partial_{u} a_2(U_0))( \nabla K)\cdot P_0\;dx - b \int_{\VR^d} |\nabla K|^2 \; dx 
     \end{split}
  \end{equation}
and testing \eqref{E:limit_K} with $\varphi = Q$ yields
\ben
    \int_{\VR^d}   (\Ca_\omega(x, \nabla K+U_0  ) - \Ca_\omega(x, U_0 ))\cdot \nabla Q \; dx 
                   = - \int_{\omega}(a_1 (U_0 ) - a_2 (U_0 )) \cdot \nabla Q \;dx.
\een

Combining these two equations we obtain
\ben
\begin{split}
    R( u, p) & =(a_1 (U_0 ) - a_2 (U_0) ) \cdot \fint_{\omega} \nabla Q \; dx \\
               & = \frac{1}{|\omega|} \int_{\VR^d} \int_0^1 \big(\partial_{u} \Ca_\omega(x, s\nabla K + U_0)(\nabla K)- \partial_{u} \Ca_\omega(x,U_0)(\nabla K)\big)\cdot P_0\;ds \;dx \\
         & + \frac{1}{|\omega|}\int_\omega (\partial_{u} a_1(U_0) -  \partial_{u} a_2(U_0))( \nabla K)\cdot P_0\;dx + \frac{1}{|\omega|} b \int_{\VR^d} |\nabla K|^2 \; dx.
\end{split}
\een
 In particular we see that for $a=1$ and $b=0$ we retrieve the formula \eqref{eq:top_formula}, that is, $R_1(u_0,p_0)+R_2(u_0,p_0) = R(u_0,p_0)$.

\section*{Conclusion}
In this paper we derived topological sensitivities for a class of quasi-linear problems under more general assumptions than previous results. Moreover, we simplified many of the previous calculations, which can be helpful when dealing with other types of nonlinear problems. In fact our analysis of $K_\eps \to K$ is not restricted to elliptic problems and is extendable to other types of equations, such as Maxwell's equation, see \cite{GanglSturm2019}.

\bibliography{topological_3D}
\bibliographystyle{plain}
\end{document}

%% file: slidesymbols.tex
\newcommand{\VR}{{\mathbf{R}}}

\newcommand{\Dsf}{\mathsf{D}}

\providecommand{\Ca}{{\cal A}}